\documentclass[12pt]{article}

\usepackage[utf8]{inputenc}
\usepackage[english]{babel}  
\usepackage[T1]{fontenc}      
\usepackage{amsfonts}
\usepackage{amsmath}
\usepackage{amssymb}
\usepackage{amsthm}
\usepackage{dsfont}

\usepackage{esint}	

\usepackage{graphics}
\usepackage{color}
\usepackage{tikz}

\newcommand{\R}{\mathbb{R}}

\newcommand{\LL}{{\rm{L}}}
\newcommand{\CC}{{\rm{C}}}

\newcommand{\WW}{{\rm{W}}}
\newcommand{\LLinf}{\LL^\infty(\R)}

\newcommand{\lt}{\left}
\newcommand{\rt}{\right}
\def\[{[\![}
\def\]{]\!]}

\newcommand{\sgn}{{\rm{sign}}}
\newcommand{\dd}{{\rm{d}}}
\newcommand{\im}{{\rm{i}}}

\newcommand{\pv}{{\rm{p.v.}}}
\newcommand{\ee}{{\rm{e}}}
\newcommand{\ii}{{\rm{i}}}

\newcommand{\essinf}{{\rm{ess~inf}}}

\newcommand{\FF}{\mathcal{F}}
\newcommand{\Hilb}{\mathcal{H}}

\newcommand{\loc}{{\rm{loc}}}
\newcommand{\comp}{{\rm{c}}}

\newcommand{\et}	{\quad\text{and}\quad}
\newcommand{\dans}	{\quad\text{in}\quad}
\newcommand{\si}	{\quad\text{if}\quad}

\newcommand{\pour}	{\quad\text{for}\quad}
\newcommand{\sur}	{\quad\text{on}\quad}

\newcommand{\dx}{\lt|\partial_x\rt|}

\newcommand{\ov}{\overline}
\newcommand{\und}{\underline}

\newcommand{\Espaceu}{\LL^\infty_\loc\lt(\R_+,\LLinf\rt)}
\newcommand{\EspSmooth}{\CC^1_{\comp}\lt([0,T), \CC^\infty_{\comp}(\R)\rt)}
\newcommand{\etar}{\eta_{\rm{r}}}
\newcommand{\etal}{\eta_{\rm{l}}}

\newcommand{\AAA}{\mathcal{A}}
\newcommand{\barl}{\hat{l}}
\newcommand{\DD}{\CC^{\infty}_{\comp}\lt(\R\rt)}
\newcommand{\Dprime}{{\mathcal{D}}'}

\newcommand{\wici}{v}

\newtheorem{theorem}{Theorem}[section]
\newtheorem{corollary}[theorem]{Corollary}
\newtheorem{lemma}[theorem]{Lemma}

\newtheorem{proposition}[theorem]{Proposition}
\newtheorem*{proposition*}{Proposition}

\newtheorem*{theorem*}{Theorem}
\newtheorem*{lemma*}{Lemma}
\newtheorem*{corollary*}{Corollary}
\newtheoremstyle{TheoremNum}
      {\topsep}{\topsep}              
      {\itshape}                      
      {}                              
      {\bfseries}                     
      {.}                             
      { }                             
      {\thmname{#1}\thmnote{ \bfseries #3}}
\theoremstyle{TheoremNum}

\theoremstyle{remark}
\newtheorem{remark}[theorem]{Remark}

\author{
  Marc Josien
  \footnote{\'Ecole des Ponts and INRIA, 6 et 8 avenue Blaise Pascal, 77455 Marne-La-Vall\'ee Cedex 2, France. Email: marc.josien@enpc.fr}
}
\title{Mathematical properties of the Weertman equation}
 
\begin{document}
 
 \maketitle
 \begin{abstract}
    We derive here some mathematical properties of the  Weertman equation and show it is the limit of an evolution equation.
    The Weertman equation is a semilinear integrodifferential equation involving a fractional Laplacian.
    In addition to this purely theoretical interest, the results proven here give a solid ground to a numerical approach that we have implemented in~\cite{TPB}.
 \end{abstract}


 \section{Introduction}
 
  \paragraph{Motivation}

    We derive here some mathematical properties of the  Weertman equation and show it is the limit of an evolution equation.
    Our motivation comes from our interest in materials science.
    The problem we consider, however classical, enjoys the following specificity that it involves the dissipative integrodifferential operator~$-\dx$ (also denoted as~$-(-\Delta)^{1/2}$), which has~$-|k|$ as Fourier symbol.
    In addition to this purely theoretical interest, the results proven here give a solid ground to a numerical approach that we have implemented in~\cite{TPB}.
    
    Our starting point is the so-called Weertman equation  (see~\cite{Rosakis}):
    \begin{align}\label{W}
	  &-\dx \eta(x)+c \eta'(x)=F'(\eta(x)) \qquad \text{for}\quad x \in \R,
      \end{align}  
      with boundary conditions
      \begin{align}
	  \label{Relimite}
	  &\lim_{x \rightarrow -\infty} \eta(x)=\etal \et \lim_{x \rightarrow +\infty} \eta(x)=\etar,
      \end{align}
     where \textit{both} the scalar~$c \in \R$ (called velocity) and the function~$\eta \in \CC^2(\R)$ are the unknowns, and where~$\etal<\etar$.
     The function~$F \in \CC^3(\R)$ is a bistable potential; namely, it satisfies
      \begin{align}
	\label{Bistable}
	  &F'(\etal)=F'(\etar)=0,\quad
	  F''(\etal)> 0, \et F''(\etar)>0.
      \end{align}
    
      From a physical point of view, Equation~\eqref{W} is a nondimensionalized form of the Weertman equation for steadily-moving dislocations in materials science  (see~\cite{Rosakis}). Dislocations are linear defects in crystals, the motion of which is responsible for the plasticity of metals. 
      From a physical standpoint, the function~$\eta$ represents a discontinuity between the local relative material displacement~$u(x,y)$ in the upper and in the lower half-spaces surrounding the glide plane on which moves the dislocation line (see Figure~\ref{FigDisloc}); see, e.g., \cite{Hirth} for details.
      In~\eqref{W}, the term~$\dx \eta$ accounts for the long-range elastic self-interactions that tend to spread the core. This repulsive interaction is counterbalanced by the nonlinear pull-back force~$F'(\eta)$, which binds together the upper and lower half-spaces. Moreover, the moving dislocation is subjected to various drag mechanisms encoded into the term~$c \eta'$.
      From a broader perspective, the function~$\eta$ can be understood as a moving phase-transformation front between the states~$\etal$ and~$\etar$  (see Figure~\ref{FigDisloc}), which are local minimizers of the potential~$F$.
      \begin{figure}[ht]
	  \begin{center}
	      \includegraphics[width=\textwidth]{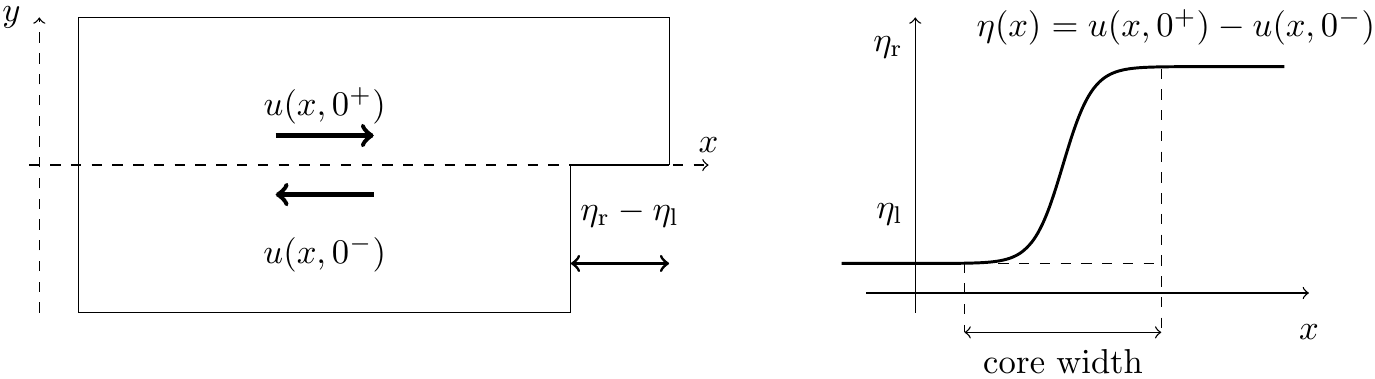}
	  \end{center}
	  \caption{\label{FigDisloc}
	   Typical shape of~$\eta(x)$, solution to~\eqref{W};  here, $u(x,y)$ is the material displacement.}
      \end{figure}

      \paragraph{Traveling wave of reaction-diffusion equation}
      Equation~\eqref{W} is a special case of the general equation 
      \begin{align}\label{S}
	\lt\{
	\begin{aligned}
	  &\mathcal{A}[\eta](x)+c\eta'(x)=0  \quad\text{for}\quad x \in \R,\\
	  &\eta(-\infty)=\etal \et \eta(+\infty)=\etar,
	\end{aligned}
	\rt.
      \end{align}
      where~$\mathcal{A}[\eta]=L\eta-F'(\eta)$ is a nonlinear operator, in which~$L$ is a diffusive operator and~$F$ is a bistable potential. 
      As is easily seen, Equation~\eqref{S} describes the traveling waves of the following reaction-diffusion equation:
      \begin{align}\label{Sd}
	\lt\{
	\begin{aligned}
	  &\partial_t u(t,x)=\mathcal{A}[u(t,\cdot)](x)&&\text{for}\quad x \in \R,\\
	  &u(0,x)=u_0(x) &&\text{for}\quad x \in \R,
	\end{aligned}
	\rt.
      \end{align}
      in the sense that, if~$u(t,x)=\eta(x-ct)$ is a traveling wave satisfying~\eqref{Sd}, then~$(\eta,c)$ solves~\eqref{S}.
      \textit{Ipso facto}, finding a solution to~\eqref{S} amounts to finding traveling waves solutions to~\eqref{Sd}.
      Natural questions thus arises:
      \begin{enumerate}
        \item[(i)]{Does Equation~\eqref{S} have one and only one solution~$(\eta,c)$? \label{i}}
        \item[(ii)]{Which properties does the solution to Equation~\eqref{S} enjoy?}
        \item[(iii)]{Is the traveling wave~$\eta(x-ct)$, for~$(\eta,c)$ solution to~\eqref{S}, an attractor of the dynamical system~\eqref{Sd}?\label{ii}}
      \end{enumerate}
      These questions have been addressed by many authors for various operators~$L$ and for bistable potential~$F$ satisfying  -most of the time- the extra condition that~$F$ does not admit any local minimum between~$\etal$ and~$\etar$.
      Other types of nonlinearities, not considered here, have attracted much attention.
      See ~\cite{Volpert_1994} for the classification of traveling waves and an overview of reaction-diffusion equations.
      
      In the seminal article~\cite{Sattinger}, Sattinger remarked that if~$(\eta,c)$ is a solution to~\eqref{S}, then~$(\eta(\cdot+\xi),c)$ is also a solution to~\eqref{S}, for arbitrary~$\xi \in \R$. Therefore, solutions to~\eqref{S} can at most be unique \textit{up to a translation}. In this regard, he introduced the notion of asymptotic stability of traveling wave and proved that the solution~$(\eta,c)$ to~\eqref{S}, if it exists, is asymptotically stable under general assumptions about the spectrum of~$\mathcal{A}$.
      
      In the celebrated article~\cite{Fife}, Fife and McLeod answered the Questions (i) and (iii) in the case where~$L$ is the Laplacian. They proved indeed that if~$F$ satisfies~\eqref{Bistable} and has no local minimum between~$\etal$ and~$\etar$, then there exists a solution~$(\eta,c)$ to~\eqref{S},  which is unique up to a translation,
      and that this solution is globally asymptotically stable. 
      Namely, for all initial conditions~$u_0$ taking values in~$[\etal,\etar]$ such that~$u_0(-\infty)=\etal$ and~$u_0(+\infty)=\etar$, there exist~$\xi \in \R$, $K>0$ and~$\kappa>0$ such that the solution~$u$ of~\eqref{Sd} satisfies
      \begin{align}\label{CvgGen}
        \lt\| u(t,\cdot) - \eta(\cdot-\xi-ct) \rt\|_{\LL^\infty(\R)} \leq K \ee^{-\kappa t},
      \end{align}
      for all~$t\in \R_+$.
      Among other important concepts, all amenable to a wide class of dissipative operators, it is observed in~\cite{Fife} that~$\mathcal{A}$ satisfies a comparison principle.
      Thus, any solution~$u(t,x)$ of~\eqref{Sd} can be squeezed between a super-solution~$w_{+1}(t,x)$ and a sub-solution~$w_{-1}(t,x)$, both at a controlled distance from~$\eta(x-\xi-ct)$.
      
      In a more recent article ~\cite{XinfuChen}, Chen combined this squeezing approach with an iterative technique. 
      Under technical assumptions about the operator~$\mathcal{A}$, he proved the global asymptotic stability of the traveling waves of~\eqref{Sd}, provided that there exists a monotonic solution~$\eta$ to~\eqref{S}. 
      In this context, a positive answer to Question (i) and technical assumptions imply, using Chen's \textit{squeezing} technique, a positive answer to Question (iii). 
      We use Chen's approach in the present article.
      
      The article~\cite{XinfuChen} also provides tools for establishing the existence and the uniqueness of a solution to~\eqref{S}.
      They have been used in~\cite{Chmaj} to positively answer to Question (i) in the case where~$L$ is the fractional Laplacian~$-\dx^\alpha$, for~$\alpha \in (0,2)$ (the latter operator has~$-|k|^\alpha$ as Fourier symbol).
      Also, in~\cite{Achleitner}, the authors have adapted Chen's squeezing technique to prove that the solution~$(\eta,c)$ to~\eqref{S} is globally asymptotically stable in the sense of~\eqref{CvgGen}, in a general framework including the case~$L=-\dx^{\alpha}$, for~$\alpha \in (1,2)$. 
      However, they underlined the fact that the case~$\alpha \leq 1$ (and in particular~$\alpha=1$), is still an open question. 
      This motivates our study.

      With an approach  different from~\cite{XinfuChen}, the existence and the uniqueness of a solution to~\eqref{S}, for~$L=-\dx^\alpha$ and~$\alpha \in (0,2)$, has been proved in~\cite{Cabre1,Cabre2} in the special case where~$c=0$  (the so-called balanced case).
      These results have been generalized by~\cite{Gui}.
      Assuming that~$F \in \CC^3(\R)$ satisfy~\eqref{Bistable} and the following extra condition:
      \begin{align}
	\label{Bis2}
	\lt\{
	  \begin{aligned}
	    &F(u) > F(\etal), \qquad \forall u \in (\etal,\etar),\\
	    &F'(u)>0 \quad\text{or}\quad F(u) >F(\etar), \qquad \forall u \in (\etal,\etar),
	  \end{aligned}
	\rt.
      \end{align}
      it is showed in~\cite[Th.\ 1.1]{Gui}  that there exist a unique~$c \in \R$ and an increasing function~$\eta \in \CC^2(\R)$, which is unique up to a translation, that solve~\eqref{W}. 
      Conditions~\eqref{Bistable} and~\eqref{Bis2} mean that the potential~$F(u)$ has two major wells in~$u=\etal$ and~$u=\etar$ (the states~$\etal$ and~$\etar$ are therefore stable), and that its behavior is controlled between these wells ; for example, the potential can have minor wells (see Figure~\ref{Fig.F}). 
      The proof of~\cite{Gui} relies on special solutions to \eqref{S} built in ~\cite{Cabre1,Cabre2}, which, by homotopy techniques, are used to find the solutions to the general case. 
      The result of~\cite{Gui} will be our starting point for proving the global asymptotic stability of the traveling waves of~\eqref{Sd} for~$L=-\dx$.
      
     Additionally, the authors of~\cite{Cabre1, Cabre2,Gui} have also studied some properties of the solutions to \eqref{S}.
      In particular, when~$L=-\dx$, they have shown that~$\eta'>0$ and that there exist constants~$B>A>0$ such that, for all~$|x| \geq 1$,
      \begin{align}
        A|x|^{-2} \leq  \eta'(x) \leq B|x|^{-2}.
	\label{O2}
      \end{align}
      See~\cite[Th.~2.7]{Cabre2} for the special case where~$c=0$ and~\cite[Prop.~3.2]{Gui} for the general case.
      Moreover, the following identity is proved (see~\cite[Prop.\ 4.1]{Gui}):
      \begin{align}
        \label{Idc1}
	  &c=\lt[F(\etar)-F(\etal)\rt] \lt( \int_{\R}\lt|\eta' \rt|^2 \rt)^{-1}.
      \end{align}
      Formula \eqref{Idc1} is useful because it provides the sign of~$c$ just by considering the values ~$F(\etal)$ and~$F(\etar)$.
	\begin{figure}[ht]
	    \begin{center}
		\includegraphics[width=6cm]{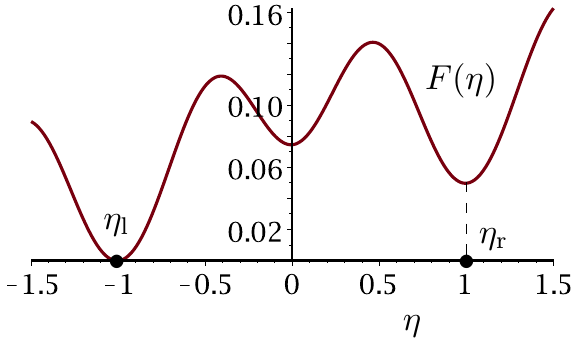}
	    \end{center}
	    \caption{\label{Fig.F} A double-well ``camel-hump'' potential~$F$, $\etal=-1$, $\etar=1$.}
	\end{figure}

       As remarked in~\cite{Cabre1}, if~$c=0$, then~\eqref{W} can be interpreted as the restriction to the boundary of an elliptic problem with Neumann boundary condition. 
	    If indeed~$u$ solves the following problem:
	    \begin{align}\label{ProbEspace}
		\lt\{
		\begin{aligned}
	        &\Delta u(x,y)+c\partial_y u(x,y) = 0 &&\dans \R\times \R_+, \\
	        &\partial_y u(x,0)=F'(u(x,0)) &&\sur \R \times \{0\},
	        \end{aligned}
	        \rt.
	    \end{align}
	    for~$c=0$, then~$\eta(x):=u(x,0)$ is a solution to~\eqref{W}. 
	    However, we stress that, when~$c\neq 0$, Equation~\eqref{ProbEspace} describes a diffusive traveling wave in the half-space.
	    In this case, \eqref{W} is \textit{not} the restriction to the boundary of the problem~\eqref{ProbEspace}, which instead reads
	    \begin{align}
		\label{CabreConsul}
	        \lt( -\Delta - c \partial_x \rt)^{1/2} \eta(x)=-F'(\eta(x)).
	    \end{align}
	    We refer to~\cite{CabreConsulMande} for a mathematical study of~\eqref{CabreConsul}.
	    
	    But, we mention for completeness that \eqref{W} is in fact the restriction to the boundary of the following elliptic equation
	    \begin{align}\label{ProbEspace2}
		\lt\{
		\begin{aligned}
	        &\Delta u(x,y) = 0 &&\dans \R\times \R_+, \\
	        &\partial_y u(x,0)+c\partial_x u(x,0)=F'(u(x,0)) &&\sur \R \times \{0\}.
	        \end{aligned}
	        \rt.
	    \end{align}
	    In a physical context, the latter is envisioned as an elastic equation in the half-plane with a nonlinear boundary condition.
	    We briefly justify it. If indeed we take the Fourier transform with respect to~$x$, denoted as~$\mathcal{F}_x$, of  the first equation of \eqref{ProbEspace2}, and if we restrict on bounded solutions, we obtain
	    \begin{align*}
	        \mathcal{F}_x\lt\{u(\cdot,y)\rt\}(k)= \ee^{-|k|y} \mathcal{F}_x\lt\{u(\cdot,0)\rt\}(k).
	    \end{align*}
	    Injecting the above information in the second equation of \eqref{ProbEspace2} then yields~\eqref{W} if we denote~$\eta(x)=u(x,0)$ (recall that~$\dx$ is an operator which has~$|k|$ as Fourier symbol).

      \paragraph{Main results}
	Our first result concerns the asymptotic expansion of  the solution~$\eta$ to~\eqref{W}. 
	The following proposition is a refinement of results of~\cite{Cabre1,Cabre2,Gui}:
	\begin{proposition}\label{PropAsymp}
	    Let~$\etal<\etar$ and~$F \in \CC^3(\R)$ satisfy~\eqref{Bistable}.
	    Assume that~$(\eta,c)$ is a solution to~\eqref{W} and \eqref{Relimite} such that~$\eta \in \CC^2(\R)$ is an increasing function satisfying~$\eta'>0$ and~\eqref{O2}.
	    Then~$\eta$ has the following asymptotes:
	  \begin{align}
	    \label{Asympt1}
	    &\eta(x)-\etar \underset{x \rightarrow + \infty}{\sim} \frac{\etal-\etar}{\pi F''(\etar)} x^{-1}, && \text{and} && \eta(x)-\etal \underset{x \rightarrow - \infty}{\sim} \frac{\etal-\etar}{\pi F''(\etal)} x^{-1}.
	  \end{align}
	\end{proposition}
	In addition to their theoretical interest, these asymptotes also allow for getting more accurate numerical approximations of~$\eta$, as shown in~\cite{TPB}. 

	 Our second result is:
	\begin{proposition}
	\label{PropIdVitesse}
	    Under the hypotheses of Proposition~\ref{PropAsymp}, $(c,\eta)$ satisfies the following identity:
	    \begin{align}
	    \label{Idc2}
	    &c=\frac{1}{\etar-\etal} \lim_{R \rightarrow + \infty} \int_{-R}^R F'(\eta).
	    \end{align}
      \end{proposition}
      The above identity is formally obtained by integrating Equation~\eqref{W} over~$\R$; we rigorously prove it.
      Notice that, by Proposition~\ref{PropAsymp} and using a Taylor expansion, $F'(\eta) \notin \LL^1(\R)$.
      
      As mentioned above in the concise form~\eqref{Sd}, we consider the following dynamical system:
      \begin{equation}\label{Wd}
	\lt\{
	  \begin{aligned}
	    &\partial_t u(t,x) + \dx u(t,x)=-F'(u(t,x)) &&\quad\text{for}\quad x\in \R, \\
	    &u(0,x)=u_0(x) &&\quad\text{for}\quad x\in \R,
	  \end{aligned}
	\rt.
      \end{equation}
      for an initial condition~$u_0 \in \LLinf$.
	We say that~$u \in \Espaceu$ is a weak solution to~\eqref{Wd} if, for all~$T>0$, for all~$\phi \in \EspSmooth$, there holds
	\begin{align}
	  \nonumber
	  &\int_0^T\int_{\R} u(t,x) \lt(-\partial_t + \dx \rt) \phi(t,x)\dd x\,\dd t\\
	  &=\int_{\R} u_0(x) \phi(0,x) \dd x
	  - \int_0^T \int_{\R} F'(u(t,x)) \phi(t,x) \dd x\,\dd t.
	  \label{FormFaible}
	\end{align}
      Our third and final result is that~\eqref{W} is the long-time limit of~\eqref{Wd}, for general initial conditions~$u_0$ with suitable behavior at infinity (see Figure \ref{Donnee_Init} for an example).
	We prove the following:
	\begin{theorem}
	  \label{ThCv}
	  Let~$\etal<\etar$, $F \in \CC^3(\R)$ satisfy~\eqref{Bistable} and~$\Delta_0>0$ be such that
	  \begin{align}\label{Puits}
	    &F''>0 \quad\text{on}\quad [\etal-\Delta_0,\etal+\Delta_0] \cup [\etar-\Delta_0,\etar+\Delta_0].
	  \end{align}
	  Assume that~$u_0 \in \LL^\infty(\R)$ takes values in~$[\etal-\Delta_0,\etar+\Delta_0]$ and satisfies
	  \begin{align}\label{Puits2}
	    \limsup_{x \rightarrow -\infty} u_0(x)<\etal+\Delta_0 \et \liminf_{x \rightarrow+\infty} u_0(x)>\etar-\Delta_0.
	  \end{align}
	  Then:
	  \begin{enumerate}
	    \item[(i)]{
	      Equation~\eqref{Wd} has a unique weak solution~$u \in \LL^\infty_{\loc}\lt(\R_+,\LL^\infty(\R)\rt)$. Moreover, for all~$T_0>0$,
	      \begin{align}
		\label{Reg_Eq1}
		u \in \CC\lt((T_0,+\infty),\CC^2(\R)\rt) \cap \CC^1 \lt((T_0,+\infty),\CC(\R)\rt).
	      \end{align}
	    }
	    \item[(ii)]{
	      In addition, for all~$t>0$ and~$x \in \R$, $u(t,x) \in [\etal-\Delta_0,\etar+\Delta_0]$.
	    }
	    \item[(iii)]{
	      Assume that~$(\eta,c)$ is a solution to~\eqref{W} and \eqref{Relimite} such that~$\eta \in \CC^2(\R)$ is an increasing function satisfying~$\eta'>0$ and
	      \begin{align}
		\label{O21}
	          \lim_{|x|\rightarrow +\infty} \eta'(x) = 0.
	      \end{align}
	      Then, there exist constants~$\kappa>0$, $K>0$ and~$\xi$ such that
	      \begin{equation}
	      \label{cv}
		\lt\|u(t,\cdot)-\eta(\cdot-c t+\xi)\rt\|_{\LL^{\infty}(\R)} \leq K \ee^{-\kappa t},
	      \end{equation}
	      for all~$t \in \R_+$.
	      In the above estimate, $\kappa$ only depends on~$F$, $\eta$ and~$\Delta_0$, whereas~$K$ and~$\xi$ also depend on~$u_0$.
	    }
	  \end{enumerate}
	\end{theorem}

	Theorem~\ref{ThCv} suggests that simulating~\eqref{Wd} is sufficient to obtain in the long time a numerical approximation of the solution~$(\eta,c)$ to~\eqref{W}. 
	In this regard, it is significant that~$c$, which is an unknown of~\eqref{W}, does \textit{not} appear in~\eqref{Wd}.
	This in particular allows for constructing an approximation of the traveling wave velocity, which is unknown before the end of the simulation.
	We refer the reader to our study~\cite{TPB}, where we explain the details of the numerical strategy, and to a forthcoming article~\cite{TPB2} for the multi-dimensional case.
	
	\begin{figure}[ht]
	    \begin{center}
		\includegraphics{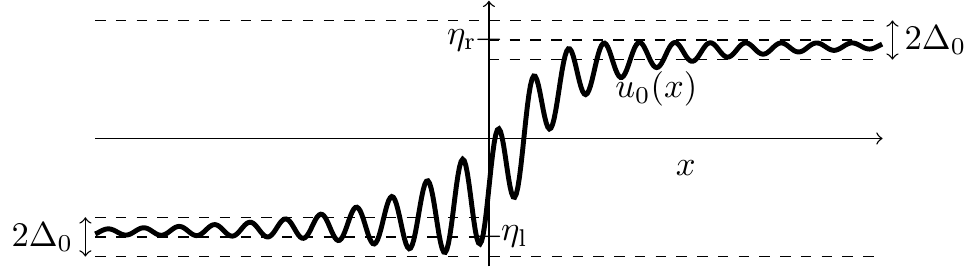}
	    \end{center}
	    \caption{A possible initial condition~$u_0$ in Theorem \ref{ThCv}.}\label{Donnee_Init}
	\end{figure}

      \paragraph{Outline}
	Our contribution is organized as follows.	
	In Section~\ref{SecNot}, we introduce notations and give essential properties of the operator~$\dx$.
	In Section~\ref{Sec_Asymp_c}, we prove Propositions~\ref{PropAsymp} and~\ref{PropIdVitesse}.
	In Section~\ref{Sec_Existe}, we justify the existence and the uniqueness of a weak solution to Equation~\eqref{Wd}, which satisfies \eqref{Reg_Eq1}, establishing thus (i) of Theorem~\ref{ThCv}.
	In Section~\ref{Sec_Converge}, we use Chen's approach for proving (ii) and (iii) of Theorem~\ref{ThCv}.
	The key ingredients are a comparison principle and specific sub-solutions and super-solutions.
	Although we could check the technical assumptions and apply Chen's theorem, we prefer to restrict Chen's proof to our special case for self-consistency and simplicity.
	
\section{Notations and definitions}\label{SecNot}

  \paragraph{Notations}
    We denote by~$\DD$ the space of smooth functions with compact supports in~$\R$ and by~$\Dprime$ the space of distributions over~$\R$. For~$u \in \DD$, we denote the Fourier transform by 
   ~$
      \FF\lt\{u\rt\}(k):=\int_{\R} \ee^{-ikx} u(x) \dd x.
   ~$
    For two functions~$u$ and~$v$, we denote by~$*$ the convolution. 
    Henceforth, the Fourier transform and the convolution are only taken with respect to the \textit{space} variable~$x$ (and never with respect to the \textit{time} variable~$t$). 
    We make use of the principal value of~$\frac{1}{x}$,
    denoted by~$\pv \lt(\frac{1}{x}\rt)$, which is the distribution defined by 
    \begin{align*}
      \lt<\pv\lt(\frac{1}{x}\rt),u\rt>&=\lim_{\epsilon \rightarrow 0^+} \lt\{\int_{\epsilon}^{+\infty} \frac{u(y)}{y} \dd y + \int_{-\infty}^{-\epsilon} \frac{u(y)}{y} \dd y \rt\},
    \end{align*}
    for~$u \in \DD$. 
 
  \paragraph{Definition and properties of the operator~$\dx$}
    For convenience, we recall some elementary properties of the operator~$\dx$.
    The Hilbert transform~$\Hilb$ of~$u \in \LL^2(\R)$ is defined by
    \begin{equation}\label{HilbertFourier}
      \Hilb\lt\{u\rt\}:=\mathcal{F}^{-1} \lt\{ -\ii ~ \sgn(k) \mathcal{F}\lt\{u\rt\}(k)\rt\}.
    \end{equation}
    It is immediate that, if~$u \in \LL^2(\R)$, then
	$
	  \Hilb^2\lt\{u\rt\}=-u.
	$
    Next, the operator~$\dx$ is defined as 
    \begin{align}
      \label{def0}
      \dx u(x):=\Hilb\lt\{u'\rt\}(x)=\FF^{-1}\lt\{ |k| \FF\lt\{u\rt\}(k) \rt\}(x),
    \end{align}
    for~$u \in \DD$.
    As~$\FF\lt\{\pv\lt(1/x\rt)\rt\}(k)=-\im\pi \sgn(k)$, the operator~$\dx$ can be rewritten as 
    \begin{align}
      \dx u(x)
      \label{def1}
      =&- \frac{1}{\pi} \int_{0}^{+\infty} \frac{u'(x+y)-u'(x-y)}{y} \dd y\\
      \label{def2}
      =& -\frac{1}{\pi} \int_{0}^{+\infty} \frac{u(x+y)-2u(x)+u(x-y)}{y^2} \dd y,
    \end{align}
    the last expression being obtained from the previous one by integrating by parts.
    We see from~\eqref{def0} that the operator~$\dx$ is symmetric and positive, like the Laplacian. But, unlike the Laplacian, it is clear from~\eqref{def2} that~$\dx u(x)$ does not only depend on~$u$ in the neighborhood of~$x$ but also on each value~$u(y)$, for~$y \in \R$; put differently, $\dx$ is \textit{non-local}.
    
    A straightforward computation yields that~$\dx \phi \in \LL^1(\R)$ whenever~$\phi \in \DD$. Hence, one can extend~$\dx$ over~$\LL^\infty(\R)$ by duality, defining~$\dx u$ as the following distribution:
      \begin{align}\label{DefDxLinfty}
        \dx u : \phi \in \DD \mapsto \int_{\R} u(y) \dx \phi(y) \dd y.
      \end{align}
    
    When~$u$ is sufficiently regular, explicit expressions of~$\dx u$ are available.
    Namely, if ~$u\in\LL^\infty(\R) \cap \CC^2(\R)$, then Expression~\eqref{def2} is valid. In particular, $\dx u \in \CC(\R) \cap \LL^\infty(\R)$. 
    The proof can be done by density of~$\DD$ in~$\CC^2_{\loc}(\R)$, using the fact that~\eqref{def2} is true for smooth functions.
    If we assume furthermore that~$u' \in \LL^1(\R)$, then Expression~\eqref{def1} is also valid; this is deduced from~\eqref{def2} by integration by parts.

  \section{Asymptotes and an identity about velocity}\label{Sec_Asymp_c}

      The proof of Proposition~\ref{PropAsymp} relies on the asymptotic behavior of Cauchy integrals (see~\cite[p.\ 267]{Musk}) and involves the following technical lemma:
      \begin{lemma}
	\label{LemGui}
	Under the hypotheses of Proposition~\ref{PropAsymp}, there holds
	\begin{align}
	  \label{O4}		
	    \eta'' \in \LL^\infty(\R).
	\end{align}
      \end{lemma}
      
      \begin{remark}
	Remark that it is also possible to establish by technical arguments that there exists a constant~$C>0$ such that, for all~$|x|>1$,
	\begin{align}
	  \label{O5}		
	  &\lt|\eta''(x)\rt|\leq C |x|^{-2} \lt(1+\ln(|x|)\rt).
	\end{align}
	 (We refer the reader to \cite{PhD} for the proof of \eqref{O5}).
	 However, \eqref{O4} is sufficient to prove Proposition~\ref{PropAsymp}.
	Yet, if~$F$ is sinusoidal, one can derive analytical solutions~$\eta$ to~\eqref{W}, as is shown in~\cite{Rosakis}, which are of the form
	\begin{align*}
	  \eta(x) = \etal + \frac{\etar-\etal}{\pi} \lt( \frac{\pi}{2} + \arctan\lt( ax\rt) \rt),
	\end{align*}
	for~$a >0$. Whence
	\begin{align*}
	  \eta''(x)= -\frac{\etar-\etal}{\pi} \frac{2a^3x}{\lt(a^2x^2+1\rt)^2} \underset{x \rightarrow +\infty}{\sim} -\frac{2(\etar-\etal)}{\pi a x^3}.
	\end{align*}
	Thus~\eqref{O5} is probably not optimal.
      \end{remark}

      We postpone the proof of Lemma~\ref{LemGui} until the end of the proof of Proposition~\ref{PropAsymp} and temporarily admit Lemma~\ref{LemGui}.
      \begin{proof}[Proof of Proposition~\ref{PropAsymp}]
	We focus on the case~$x \rightarrow + \infty$. Provided that
	\begin{align}
	  \label{Pretention}
	  x  \dx \eta(x)  \underset{x\rightarrow +\infty}{\longrightarrow}  \frac{1}{\pi} \int_{-\infty}^{+\infty} \eta'(y) \dd y= \frac{1}{\pi} \lt(\etar-\etal\rt),
	\end{align}
	then, using~\eqref{W}, \eqref{Relimite}, \eqref{Bistable} and ~\eqref{O2}, we get by Taylor expansion
	\begin{align*}
	  F''(\etar) \lt( \eta(x)-\etar \rt)
	  \underset{x \rightarrow + \infty}{\sim}
	  -\frac{1}{\pi} x^{-1} \lt( \etar-\etal\rt),
	\end{align*}
	which is~\eqref{Asympt1}. 
	
	Let us now prove~\eqref{Pretention}. By assumption, $\eta' \in \CC^1(\R)$, and by~\eqref{O2}, we have~$\eta' \in \LL^{1}(\R)$. As a consequence, there holds
	\begin{equation*}
	    x\dx \eta(x)=\frac{x}{\pi} \lim_{\epsilon \rightarrow 0^+} \lt( \int_{-\infty}^{x-\epsilon} + \int_{x+\epsilon}^{+\infty} \rt) \frac{\eta'(y)}{x-y} \dd y.
	\end{equation*}
	Let~$R>0$ and~$x>2R$. We split the integral into three parts 
	\begin{align}
	\nonumber
	 \pi x\dx \eta(x)=&
	  \int_{-\infty}^R \frac{\eta'(y)}{1-y/x} \dd y
	  + x\lt(\int_{R}^{x-R} + \int_{x+R}^{+\infty} \rt) \frac{\eta'(y)}{x-y} \dd y\\
	  &{}+ x\lim_{\epsilon \rightarrow 0^+}\lt(\int_{x-R}^{x-\epsilon}+\int_{x+\epsilon}^{x+R}\rt)\frac{\eta'(y)}{x-y} \dd y.
	  \label{3parts}
	\end{align}
	The first right-hand term in \eqref{3parts} is dealt with the dominated convergence theorem, the second one avoids the singularity of~$|x-y|^{-1}$ and is bounded thanks to~\eqref{O2}, and the third one is on the singularity of~$|x-y|^{-1}$ and is controlled thanks to \eqref{O4} and \eqref{O2}.
	
	As~$\eta' \in\LL^{1}(\R)$ and since (recall that~$x>2R$)
	\begin{align*}
        \lt|\frac{\eta'(y)}{1-y/x}\rt|  \leq 2 \lt|\eta'(y)\rt|
        \quad\text{if}\quad
        y<R,
        \quad\text{and}\quad
        \frac{\eta'(y)}{1-y/x} \underset{x \rightarrow + \infty}{\longrightarrow} \eta'(y),
	\end{align*}
	then, by the dominated convergence theorem 
	\begin{align}\label{first_part}
        \int_{-\infty}^{R} \frac{\eta'(y)}{1-y/x} \dd y \underset{x \rightarrow + \infty}{\longrightarrow} \int_{-\infty}^R \eta'(y) \dd y.
	\end{align}
	Next, we split the second integral of~\eqref{3parts} into three parts. Invoking~\eqref{O2}, we deduce that, as~$|x|>2R$,
	\begin{align}
	  \nonumber
	  &\lt|\lt( \int_{R}^{x-R} + \int_{x+R}^{+\infty} \rt) \frac{\eta'(y)}{x-y} \dd y\rt|
	  \\
	  \nonumber
	  &\leq C x^{-1} \int_{R}^{x/2} \lt|\eta'(y)\rt|\dd y 
	  + C R^{-1}
	  \lt(\int_{x/2}^{x-R}+\int_{x+R}^{+\infty}\rt)\lt|\eta'(y)\rt|\dd y
	  \\
	  &\leq C x^{-1} \int_{R}^{x/2} \frac{\dd y}{y^2} + C R^{-1}
	  \lt(\int_{x/2}^{x-R}+\int_{x+R}^{+\infty}\rt)\frac{\dd y}{y^2}
	  \nonumber
	  \\
	  &\leq CR^{-1} x^{-1}.
	  \label{second_part}
	\end{align}
	Last, we split the last integral of \eqref{3parts} into two parts, namely:
	\begin{align}
	  \nonumber
	  \lt|\lt(\int_{x-R}^{x-\epsilon}+\int_{x+\epsilon}^{x+R}\rt)\frac{\eta'(y)}{x-y} \dd y \rt|
	  =&\lt|\lt(\int_\epsilon^{x^{-2}}+\int_{x^{-2}}^R\rt) \frac{\eta'(x-z)-\eta'(x+z)}{z} \dd z\rt|.
	\end{align}
	The first part of the right-hand side of the above equation is dealt with by using \eqref{O4}, and the second one by using \eqref{O2}.
	Whence, as~$x>2R$ and for~$x^{-2} \geq \epsilon$,
	\begin{align*}
	  \lt|\lt(\int_{x-R}^{x-\epsilon}+\int_{x+\epsilon}^{x+R}\rt)\frac{\eta'(y)}{x-y} \dd y \rt|
	  &\leq C x^{-2} \int_{x^{-2}}^{R} z^{-1}\dd z + C \int_{\epsilon}^{x^{-2}} \dd z
	  \\
	  &\leq C x^{-2}\lt(\ln(Rx^2) +1 \rt).
	\end{align*}
	Therefore
	\begin{align}
	     \lt|x\lim_{\epsilon \rightarrow 0^+}\lt(\int_{x-R}^{x-\epsilon}+\int_{x+\epsilon}^{x+R}\rt)\frac{\eta'(y)}{x-y} \dd y  \rt|\leq C x^{-1}\lt(\ln(Rx^2) +1 \rt).
	  \label{third_part}
	\end{align}

	Convergence~\eqref{first_part} and Estimates~\eqref{second_part} and~\eqref{third_part} finally yield 
	\begin{align*}
	&\limsup_{x \rightarrow + \infty} 
	\lt|\pi x  \dx \eta(x)  - \int_{-\infty}^{+\infty} \eta'(y) \dd y\rt|
	\leq \lt| \int_{R}^{+\infty} \eta'(y)\dd y \rt|+ CR^{-1},
	\end{align*}
	which, thanks to~\eqref{O2}, implies~\eqref{Pretention} upon letting~$R \rightarrow + \infty$.
      \end{proof}

      We then proceed with the:
      \begin{proof}[Proof of Lemma~\ref{LemGui}]
      We first remark that if~$g \in\LL^{2}(\R)$ and if
	\begin{equation}
	  \label{Asympt_Eq1}
	  -\dx h(x) + c h'(x)=g(x),
	\end{equation}
	then~$h' \in\LL^{2}(\R)$. Indeed, the Fourier transform turns~\eqref{Asympt_Eq1} into
	\begin{equation*}
	  (-|k|+c\ii k)\mathcal{F}\{h\}(k) = \mathcal{F}\{g\}(k).
	\end{equation*}
	Therefore, $k\mathcal{F}\{h\}(k) \in\LL^{2}(\R)$, whence~$h' \in\LL^{2}(\R)$. We use this result to prove that~$\eta'' \in \LL^{\infty}(\R)$. 
	
	Upon differentiating~\eqref{W}, we obtain
	\begin{equation}\label{LemEq1}
	    -\dx \eta'(x)+c \eta''(x)= F''(\eta(x))\eta'(x).
	\end{equation}
	As~$\eta \in\LL^{\infty}(\R)$, $F\in \CC^{3}(\R)$ and~$\eta' \in\LL^{2}(\R)$ (thanks to~\eqref{O2}),  then the right-hand side of~\eqref{LemEq1} is in~$\LL^{2}(\R)$. Therefore~$\eta'' \in\LL^{2}(\R)$. Differentiating~\eqref{LemEq1} yields
	\begin{equation}\label{LemEq2}
	    -\dx \eta''(x)+c\eta'''(x) = F''(\eta(x))\eta''(x)+ F'''(\eta(x)) \lt(\eta'(x)\rt)^2.
	\end{equation}
	As~$\eta \in \LL^{\infty}(\R)$, $F \in \CC^{3}(\R)$, $\eta'' \in \LL^{2}(\R)$, and, thanks to~\eqref{O2}, $\lt(\eta'\rt)^2 \in \LL^{2}(\R)$, then the right-hand side of~\eqref{LemEq2} is in~$\LL^{2}(\R)$. Therefore~$\eta'''\in\LL^{2}(\R)$. As a consequence, since~$\eta'' \in \LL^{2}(\R)$, we deduce by Sobolev injection that~$\eta'' \in \LL^\infty(\R)$,
	whence \eqref{O4}.	
      \end{proof}

      We now focus on Proposition~\ref{PropIdVitesse}.
      Both Identities~\eqref{Idc1} and~\eqref{Idc2} are formally obtained by testing~\eqref{W} against a certain function~$g$, namely~$g=\eta'$ for~\eqref{Idc1}, and~$g=1$ for~\eqref{Idc2}. We justify below this formal integration.

      \begin{proof}[Proof of Proposition~\ref{PropIdVitesse}]
	Let~$R>2$. We integrate~\eqref{W} over~$[-R,R]$:
	\begin{equation}\label{etoile}
	  -\int_{-R}^{R} \dx \eta(x)\dd x + c \lt( \eta(R)-\eta(-R)\rt)=\int_{-R}^R F'(\eta(x))\dd x.
	\end{equation}
	Thus, Identity~\eqref{Idc2} stems from~\eqref{etoile}, provided that
        \begin{equation}\label{Pretention2}
          \lim_{R \rightarrow + \infty}\int_{-R}^R \dx \eta(x)\dd x =0.
        \end{equation}

        We prove~\eqref{Pretention2} using~\eqref{O2} and~\eqref{O4}.
        As~$\eta' \in\LL^{1}(\R) \cap \CC^{1}(\R)$, 
        there holds
        \begin{align}
	  \nonumber
          \dx \eta(x)
          =& \frac{1}{\pi}\int_0^{+\infty} \frac{\eta'(x-y)-\eta'(x+y)}{y}\dd y
          \\
	  =& \frac{1}{\pi} \int_{|y|<R}  \frac{\lt(\eta'(x-y) - \eta'(x) \rt)}{y}\dd y+ \frac{1}{\pi} \int_{|y|>R} \frac{\eta'(x-y)}{y}\dd y.
	  \label{eq:Star}
        \end{align}
        Remark that
        \begin{align*}
	    \int_{|y|>R} \lt|\frac{\eta'(x-y)}{y}\rt|\dd y \leq 2 R^{-1}\lt\|\eta'\rt\|_{\LL^{1}(\R)},
        \end{align*}
        and that, using~\eqref{O4},
        \begin{align*}
	    \int_{|y|<R} \lt| \frac{\lt(\eta'(x-y) - \eta'(x) \rt)}{y} \rt|\dd y
	    \leq & 2 R \lt\|\eta'' \rt\|_{\LL^{\infty}(\R)} \leq CR.
        \end{align*}
        Therefore, integrating~\eqref{eq:Star} thanks to Fubini's theorem yields
        \begin{align}
	    \int_{-R}^R \dx \eta(x)\dd x
	    =&\frac{1}{\pi}\int_{|y|<R} \frac{\lt(\eta(R-y)-\eta(R) \rt)-\lt(\eta(-R-y)-\eta(-R) \rt)}{y}\dd y
	    \nonumber
	    \\
	    \nonumber
	    &+\frac{1}{\pi}\int_{|y|>R} \frac{\eta(R-y)-\eta(-R-y)}{y}\dd y\\
	    \label{DerdesDers}
	    =&:T_1+T_2.
        \end{align}
	First, we bound~$T_1$. If~$|y|<R/2$, thanks to~\eqref{O2}, we obtain
	\begin{align*}
	  &\lt|\eta(R-y)-\eta(R)\rt|\leq C |y|R^{-2}
	  \quad\text{and}\quad
	  \lt|\eta(-R-y)-\eta(-R)\rt|\leq C |y|R^{-2}.
	\end{align*}
	As a consequence,
	\begin{align}
	\label{SurT11}
	  &\int_{|y|<R/2} \lt|\frac{\lt(\eta(R-y)-\eta(R) \rt)-\lt(\eta(-R-y)-\eta(-R) \rt)}{y}\rt|\dd y \leq C R^{-1}.
	\end{align}
	Note that, as underlined in~\cite{Gui}, a consequence of \eqref{O2} is that
	\begin{align}
	    \lt\{
	    \label{O1}
	    \begin{aligned}
		&A|x|^{-1} \leq \etar-\eta(x) \leq B |x|^{-1}, 	\si x>1,\\
		&A|x|^{-1} \leq \eta(x)-\etal \leq B |x|^{-1}, 	\si x<-1.
	    \end{aligned}
	    \rt.
	\end{align}
	Therefore, if~$|y|<R$
	\begin{align*}
	  &\lt|\eta(R-y)-\eta(R)\rt| \leq \frac{C}{R-|y|+1}
	  \quad\text{and}\quad
	  \lt|\eta(-R-y)-\eta(-R)\rt| \leq \frac{C}{R-|y|+1}.
	\end{align*}
	Whence
	\begin{align}
	  \nonumber
	  &\int_{R/2<|y|<R} \lt|\frac{\lt(\eta(R-y)-\eta(R) \rt)-\lt(\eta(-R-y)-\eta(-R) \rt)}{y}\rt|\dd y
	  \\
	  \label{SurT12}
	  &\leq\frac{C}{R} \int_{R/2}^R \frac{\dd y}{R+1-y}
	  \leq C R^{-1}\ln(R).
	\end{align}
	We deduce from~\eqref{SurT11} and~\eqref{SurT12} that
	\begin{equation}
	\label{SurT1}
	  \lt|T_1(R)\rt|\leq C R^{-1} (1+\ln (R)).
	\end{equation}
	Thanks to~\eqref{O2} and since~$\eta \in\LL^{\infty}(\R)$, if~$|y|>R$, we have
	\begin{equation*}
	  \lt|\eta(R-y)-\eta(-R-y)\rt|\leq C \min\lt\{R \lt(|y|-R\rt)^{-2}, 1 \rt\}.
	\end{equation*}
	Whence, splitting~$T_2$ into two parts,
	\begin{align}
	  \nonumber
	  \lt| T_2(R) \rt|
	  \leq & C \lt(\int_{R}^{R+\sqrt{R}}\frac{\dd y}{y}+\int_{R+\sqrt{R}}^{+\infty} \frac{R}{y(y-R)^2}\dd y\rt)
	  \\
	  \label{SurT2}
	  \leq & C \frac{\sqrt{R}}{R} + C\int_{\sqrt{R}}^{+\infty}\frac{\dd z}{z^2}
	  \leq  C R^{-1/2}.
	\end{align}
	Bearing~\eqref{DerdesDers} in mind, we observe that~\eqref{SurT1} and~\eqref{SurT2} imply~\eqref{Pretention2}.
      \end{proof}

  \section{Existence, uniqueness and regularity of the solution to the evolution equation~\eqref{Wd}}\label{Sec_Existe}

    We now justify the existence, the uniqueness and the regularity of a weak solution~$u$ to~\eqref{Wd}. 
    We proceed in the classical way; as the methods as well as the type of results are well-known, we only give a few hints of proofs. We refer the interested reader to~\cite{PhD} for some technical details and extra materials about the proofs, and to~\cite{Cazenave} for a reference on evolution equations involving m-dissipative operators.
    
        Using the Fourier transform, the solution to the homogeneous linear equation
	\begin{align}\label{EqHom}
	    &\partial_t u(t,x)+\dx u(t,x)=0
	    \pour x \in \R, && \et u(0,x)=u_0(x)
	\end{align}
         is given by~$u(t,x)=\lt\{K_t*u_0\rt\}(x)$, where the kernel~$K_t$ is defined by
        \begin{align}\label{DefKt}
          K_t(x)=
	    \begin{aligned}
	      &\frac{t}{\pi (t^2+x^2)} \quad \si t>0 
	      \et
	      K_0=\delta_0,\\
	    \end{aligned}
        \end{align}
        the Fourier transform of which is~$\ee^{-|k|t}$.
	Before getting to the inhomogeneous linear equation, we underline some interesting properties of the kernel~$K_t$. 
	First, for all~$t \geq 0$, $K_t$ is a probability measure.
	Moreover, for all~$t>0$, $K_t$ is a smooth function. In particular, the space derivative of~$K_t$ satisfies
      \begin{align}
	\label{DerivKt}
	\lt\|\frac{\dd}{\dd x} K_t\rt\|_{\LL^1(\R)} \leq C t^{-1},
      \end{align}
      where~$C$ is a universal constant. In all these aspects, $K_t$ is similar to the Gaussian kernel~$\mathcal{K}_t(x)=\ee^{-\frac{x^2}{2t^2}}/(t\sqrt{2\pi})$.

      The semi-group generated by~$K_t$ allows for solving the inhomogeneous equation
      \begin{align}
        \label{EqInhom}
        \lt\{
	  \begin{aligned}
	    &\partial_t u(t,x)+\dx u(t,x)=g(t,x) 	    &&\quad\text{for}\quad x \in \R,
	    \\
	    &u(0,x)=u_0(x) 	    &&\quad\text{for}\quad x \in \R.
	  \end{aligned}
        \rt.
      \end{align}
	Indeed, let~$T>0$, $u_0 \in \LL^\infty(\R)$ and~$g\in \LL^\infty([0,T] \times \R)$.
	Then there exists a unique weak solution~$u \in \LL^\infty([0,T] \times \R)$ to~\eqref{EqInhom} in the sense that, for all~$\phi \in \EspSmooth$, the following identity holds:
        \begin{align}
	  \nonumber
	  &\int_0^T \int_{\R} u(t,x) \lt(-\partial_t + \dx \rt) \phi(t,x) \dd x \dd t - \int_{\R} u_0(x) \phi(0,x) \dd x\\
	  &= \int_0^T \int_{\R} g(t,x) \phi(t,x) \dd x \dd t.
	  \label{FormFaibleGen}
	\end{align}
        This solution can be written thanks to the Duhamel formula as
        \begin{align}\label{EqInhomExp}
          u(t,x)= K_t * u_0(x) + \lt\{ \int_0^t K_{t-s}*g(s,\cdot) \dd s \rt\}(x),
        \end{align}
	with the convention that~$K_0 * u_0=u_0$, even if~$u_0$ is not regular. The existence of a solution~$u$ to~\eqref{EqInhom} is a consequence of the fact that~\eqref{EqInhomExp} is well-defined;  its uniqueness is showed using the adjoint problem of~\eqref{EqInhom}.

	We now turn to the semi-linear equation~\eqref{Wd}. Let~$F \in \WW^{2,\infty}(\R)$ and~$u_0 \in \LL^\infty(\R)$. Then there exists a unique weak solution~$u \in \LL_\loc^{\infty}\lt(\R_+,\LL^\infty(\R)\rt)$ to~\eqref{Wd}. Moreover, $u$ can be expressed as
        \begin{align}
          \label{Explicit_u}
          u(t,x)=\lt\{K_t * u_0\rt\} (x) - \lt\{ \int_0^t K_{t-s} * F'(u(s,\cdot)) \dd s \rt\}(x).
        \end{align}
      The proof is done by a classical fixed-point argument  on~\eqref{Explicit_u} (see for example~\cite[Sec.\ 4.3 p.\ 56]{Cazenave}).
      
      Finally, we justify that the evolution equation~\eqref{Wd} has a regularizing effect; in other words, the weak solution to~\eqref{Wd} becomes instantly a classical solution.
        Assume indeed that~$F \in \CC^3(\R) \cap \WW^{3,\infty}(\R)$ and~$u_0 \in \LL^\infty(\R)$, and let~$u \in \LL^{\infty}_{\loc}\lt(\R_+,\LLinf\rt)$ be the weak solution to~\eqref{Wd}. Then, for all~$T_0>0$, we have \eqref{Reg_Eq1}.
        Therefore, for all~$t>0$, there holds
        \begin{align}
	  \label{Reg_Eq2}
          \partial_t u(t,x)+\dx u(t,x)=-F'(u(t,x)),
        \end{align}
        in the strong sense. Finally
	\begin{align}
	  \label{Reg_Eq3}
	  u \in \CC\lt(\lt[0,+\infty\rt[,\text{weak-$*$-}\LL^\infty(\R) \rt).
	\end{align}
	The proof of~\eqref{Reg_Eq1} relies on an iterative argument based on~\eqref{Explicit_u}, using the fact that, for all~$t>0$, $K_t$ is a smooth probability measure that satisfies~\eqref{DerivKt}.
	Last, a straightforward adaptation of the proof of~\cite[Ex.\ 4.24 p.\ 126]{Brezis} yields~\eqref{Reg_Eq3}.

        Note that, similarly to~\cite[Th.\ 2.9.]{Achleitner}, $t\mapsto u(t,\cdot)$ may not be continuous at~$0$ in~$\LL^\infty(\R)$. If indeed~$u_0$ is discontinuous, then~$u(t,\cdot)$ cannot tend to~$u_0$ in~$\LL^\infty(\R)$ when~$t\rightarrow 0$ because, invoking~\eqref{Reg_Eq1}, ~$u(t,\cdot)$ is a continuous function for all~$t>0$.

 \section{Convergence of the evolution equation~\eqref{Wd} to the Weertman equation~\eqref{W}}\label{Sec_Converge}
    
    In this section, we prove (ii) and (iii) of Theorem~\ref{ThCv}. 
    The proof can be summarized in two steps: first, we show that~\eqref{Wd} satisfies a comparison principle, then we use Chen's method of squeezing, establishing respectively (ii) and (iii) of Theorem~\ref{ThCv}. 
    For the sake of self-consistency, simplicity and conciseness (Chen's theory being quite general), we prefer to restrict the whole proof of ~\cite[Th.\ 3.1]{XinfuChen} to our specific case rather than to check that the hypotheses of Chen's theory are satisfied (precisely Hypotheses (A1), (A2), (A3), (B1), (B2) and (B3) of~\cite{XinfuChen}, which are indeed satisfied in our case). 
    
      We henceforth assume that~$F \in \CC^3(\R) \cap \WW^{3,\infty}(\R)$ satisfies~\eqref{Bistable} and~\eqref{Puits}. 
      We introduce the non-linear operator~$\AAA[u]:=-\dx u -F'(u)$, of which we now discuss some immediate properties. 
      By the results of Section~\ref{Sec_Existe}, $\AAA$ generates a semi-group on the Banach space~$\LLinf$.~$\AAA$ is translation invariant; namely, for all~$h \in \R$, and for any function~$u(x)$, there holds
      \begin{align}\label{InvarTrans}
	\AAA[u(h+\cdot)](x)=\AAA[u](x+h), \qquad \forall x \in \R.
      \end{align}
      Moreover, $\AAA$ maps constant functions to constant functions; namely
      \begin{align*}
	\AAA[\alpha\cdot1]=-F'(\alpha)\cdot1,
      \end{align*}
      for all~$\alpha \in \R$, where~$1$ above denotes the function identically equal to~$1$. 
      
      The operator~$\AAA$ satisfies the following comparison principle:
      \begin{proposition}\label{PropCompare}
	Let~$F \in \WW^{2,\infty}(\R)$. Let~$\ov{u}$ and~$\und{u} \in \Espaceu$ be such that
	\begin{align}
	  \label{HypoComp}
	  \lt\{
	    \begin{aligned}
	      &\partial_t \und{u}(t,x) - \mathcal{A}[\und{u}(t,\cdot)](x)=\und{g}(t,x) \leq 0,\\
	      &\partial_t \ov{u}(t,x)-\mathcal{A}[\ov{u}(t,\cdot)](x)=\ov{g}(t,x) \geq 0,\\
	    \end{aligned}
	  \rt.
	\end{align}
	 where~$\und{g}$ and~$\ov{g} \in \Espaceu$, and~$\und{u}(0,\cdot) \leq \ov{u}(0,\cdot)$ with~$\und{u}(0,\cdot) \neq \ov{u}(0,\cdot)$ on a non-negligible set. Then, for almost every~$t>0$, $x \in \R$, there holds
	\begin{align}\label{Comp1}
	  \und{u}(t,x)< \ov{u}(t,x).
	\end{align}
      \end{proposition}
      
      \begin{remark}
	\label{RqExiste}
      	Proposition~\ref{PropCompare} has an immediate corollary: 
      	Assume that~$u_0 \in \LLinf$ takes values in~$[\etal-\Delta_0,\etar+\Delta_0]$ and let ~$u \in \Espaceu$ be the unique solution to~\eqref{Wd} with the initial condition~$u_0$.
      	By \eqref{Bistable}, $\ov{u}:=x \mapsto \etar+\Delta_0$ and~$\und{u} :x \mapsto \etal-\Delta_0$ are respectively supersolutions and subsolutions to~\eqref{Wd}. Therefore, Proposition~\ref{PropCompare} implies that~$u(t,x) \in [\etal-\Delta_0,\etar+\Delta_0]$, for almost every~$t \in \R_+$ and~$x\in \R$, establishing thus (ii) of Theorem~\ref{ThCv}.
      \end{remark}
      
      Proposition~\ref{PropCompare} is shown thanks to the Duhamel Formula~\eqref{EqInhomExp} and Gr\"onwall's Lemma.
      \begin{proof}
	Let~$M>\lt\|F''\rt\|_{\LL^\infty(\R)}$. We set
	\begin{align}\label{DefW}
	  \wici(t,x):=\ee^{Mt} \lt(\ov{u}(t,x)-\und{u}(t,x)\rt),
	\end{align}
	and prove that~$\wici \geq 0$. 
	In view of~\eqref{HypoComp}, we have
	\begin{align*}
	  \partial_t \wici(t,x) + \dx \wici(t,x) = M \wici(t,x) - \ee^{Mt} \lt( F'(\ov{u}(t,x))-F'(\und{u}(t,x))+g(t,x\rt),
	\end{align*}
	where~$g(t,x):=e^{Mt}\lt(\ov{g}(t,x)-\und{g}(t,x)\rt) \geq 0$.
	Since the right-hand side of the latter equation is in~$\Espaceu$, then, using~\eqref{EqInhomExp}, one can express~$\wici(t,x)$ as
	\begin{align}
	  \nonumber
	  \wici(t,x)=& \lt\{K_t *\wici(0,\cdot) \rt\}(x)+ \Bigg\{ \int_0^t K_{t-s}*\Big[M \wici(s,\cdot)
	  \\
	  \label{Eq_SurW}
	  &-\ee^{Ms} \lt(F'(\ov{u}(s,\cdot))-F'(\und{u}(s,\cdot))+g(s,\cdot) \rt) \Big] \dd s\Bigg\}(x).
	\end{align}
	We introduce~$\wici_-(t):=-\essinf \lt\{ \min(0,\wici(t,x)), x \in \R \rt\}$. By Taylor expansion, for almost every~$(s,y) \in [0,t] \times \R$, there holds
	\begin{align}
	  \nonumber
	  M \wici(s,y) -\ee^{Ms} \lt(F'(\ov{u})-F'(\und{u}) \rt)(s,y) 
	  \geq & M \wici(s,y) - \lt\|F''\rt\|_{\LLinf} \ee^{Ms} \lt|\ov{u}(s,y)-\und{u}(s,y)\rt| 
	  \\
	  \nonumber
	  \geq & M\wici(s,y)-M|\wici(s,y)|
	  \\
	  \geq & -2M \wici_-(s).
	  \label{Youpi}
	\end{align}
	Therefore, using~\eqref{Eq_SurW}, since~$K_t$, $g$ and~$\wici(0,\cdot)$ are nonnegative, and since~$K_t$ is a probability measure for all~$t\geq0$, we obtain, for almost every~$t\in \R_+$ and~$x \in \R$,
	\begin{align*}
	  -\wici(t,x)) \leq 2M \int_0^t \wici_-(s)  \dd s,
	\end{align*}
	whence
	\begin{align}\label{PourGron}
	   \wici_-(t)\leq 2M \int_0^t \wici_-(s) \dd s.
	\end{align}
	Since~$\und{u}$, $\ov{u} \in \Espaceu$, then, $\wici_- \in \LL^{\infty}_{\loc}\lt(\R_+\rt)$. 
	Hence, by Gr\"onwall's Lemma, we deduce from~\eqref{PourGron} that~$\wici_-(t)=0$, for almost every~$t>0$.
	Injecting this information in~\eqref{Youpi}, and next in~\eqref{Eq_SurW}, yields
	$
	  \wici(t,x) \geq \lt\{K_t * \wici(0,\cdot)\rt\}(x).
	$
	As a consequence, as~$K_t$ is positive if~$t>0$ and as~$\wici(0,\cdot)=\ov{u}(0,\cdot)-\und{u}(0,\cdot)$ is nonnegative and positive on a non-negligible set, we deduce that~$\wici(t,x) >0$,
	for almost every~$t>0$ and~$x \in \R$.
	This implies~\eqref{Comp1}.
      \end{proof}

	Then, we establish a stronger version of the comparison principle, the proof of which mimicks that of Proposition~\ref{PropCompare}:
      \begin{corollary}
        Under the assumptions of Proposition~\ref{PropCompare}, there exists a positive decreasing function~$\rho$ such that, for all~$R>1$,
	\begin{align}\label{ExistRho}
	  \essinf_{x \in [-R,R]} \lt(\ov{u}(1,x) - \und{u}(1,x)\rt) \geq \rho(R) \int_{0}^1 \lt( \ov{u}(0,y)-\und{u}(0,y) \rt)\dd y.
	\end{align}
      \end{corollary}

      \begin{proof}
	Introducing~$\wici$ defined by \eqref{DefW}, and using \eqref{Eq_SurW} and \eqref{Youpi}, we obtain
	\begin{align*}
	    \wici(t,x) \geq \lt\{K_t *\wici(0,\cdot) \rt\}(x),
	\end{align*}
	since~$\wici$ is nonnegative.
	By definition of~$K_t$ and of~$\wici$, it implies \eqref{ExistRho}.
      \end{proof}

    The proof of (iii) of Theorem~\ref{ThCv} is done after~\cite[Th.\ 3.1]{XinfuChen}, the proof of which we restrict here to our particular case. 
    By the previous steps, we already know that there exists a unique weak solution~$u(t,x)$ to~\eqref{Wd} (see Remark~\ref{RqExiste}), and we aim at establishing~\eqref{cv}.
    
    Namely, we build special sub-solutions and super-solutions to~\eqref{Wd} that are based on the existing solution to~\eqref{W} (see Lemma~\ref{Lem22} below).
    Then, we prove that both the vertical and the horizontal distances between these solutions are controlled (respectively~$2\delta_j$ and~$2l_j$ on Figure~\ref{FigSqueeze}, see also Lemma~\ref{Lem33} below).
    Using the fact that \eqref{Wd} is an autonomous system, we use the established control to iteratively build successive sub-solutions~$w_{-1}^j$ and super-solutions~$w_{+1}^j$ surrounding the actual solution~$u(t_j,x)$ to \eqref{Wd}.
    At each step~$j$, the distance between these sub-solutions and super-solutions is lowered.
    Thus, the solution~$u$ is \textit{squeezed} between these sub-solutions and super-solutions.
    As a consequence, when~$t$ goes to infinity, the solution~$u(t,\cdot)$ tends, up to a translation, to the solution to \eqref{W}.
    Because of the iterative nature of the squeezing, this convergence is achieved with exponential speed.
    \begin{figure}[ht] 
	\label{FigSqueeze}
	\begin{center}
		\includegraphics[width=\textwidth]{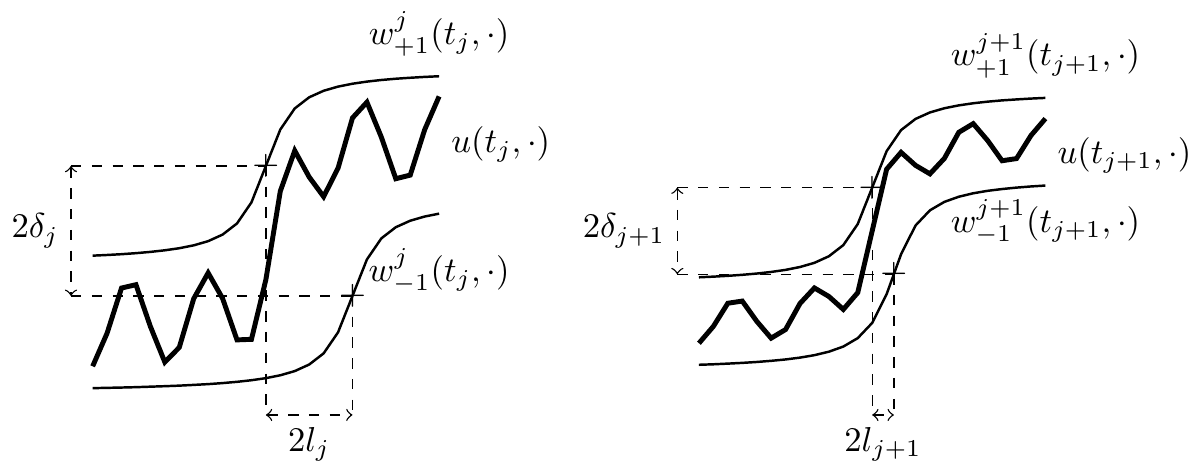}
        \end{center}
        \caption{Squeezing of~$u(t,x)$ solution to \eqref{Sd}.}
      \end{figure}
    
    \begin{lemma}[Lemma 2.2 of~\cite{XinfuChen}]
    \label{Lem22}
      Under the hypotheses of Theorem~\ref{ThCv}, let~$\Delta_1<\Delta_0$. Then, there exist positive constants~$\sigma$ and~$\beta$ 
      such that, for all~$\delta \in (0, \Delta_1)$ and~$l \in \R$, the functions~$w_{-1}(t,x)$ and~$w_{+1}(t,x)$ defined by
      \begin{align}
	\label{eq:Def_w}
          w_{i}(t,x)=\eta\lt(\zeta_{i}(t,x)\rt)+i\delta \ee^{-\beta t} && \pour i \in \{-1,+1\},
      \end{align}
      where
       \begin{align}
	    \label{DefZeta}
	    \zeta_i(t,x)=x-ct+il+i\sigma \delta \lt[1-\ee^{-\beta t}\rt]  && \pour i \in \{-1,+1\},
 	\end{align}
      are respectively a sub-solution and a super-solution to~\eqref{Wd}.
    \end{lemma}
    \begin{proof}[Proof of Lemma~\ref{Lem22}]
	As~$\AAA$ is invariant by translation, the variable~$l$  in the definition of~$\zeta$ plays no role. Hence, we take~$l=0$ in the proof below.
	We also impose for the moment~$\sigma\leq 1$.

      A straightforward computation yields
      \begin{align*}
        \lt(\partial_t + \dx \rt)w_{i}(t,x)=\lt(i \sigma\beta\delta\ee^{-\beta t/2}-c\rt)\eta'\lt(\zeta_i(t,x)\rt)-i\delta\beta\ee^{-\beta t} + \dx \eta\lt(\zeta_i(t,x)\rt),
      \end{align*}
      and, as~$\eta$ satisfies~\eqref{W},
      \begin{align*}
        \lt(\partial_t + \dx \rt)w_{i}(t,x)=i\sigma\beta\delta\ee^{-\beta t}\eta'\lt(\zeta_i(t,x)\rt) -i\beta\delta\ee^{-\beta t} - F'\lt( \eta\lt(\zeta_i(t,x)\rt)\rt).
      \end{align*}
      Thus
      \begin{align*}
	\lt(\partial_t - \AAA\rt)w_{i}(t,x)=&i\beta\delta \ee^{-\beta t}\lt[\sigma \eta'\lt(\zeta_i(t,x)\rt)-1\rt]
        +F'\lt(w_{i}(t,x)\rt)-F'\lt(\eta\lt(\zeta_i\lt(t,x\rt)\rt)\rt).
      \end{align*}
      By Taylor-Lagrange expansion, there exists a convex combination~$\theta^i(t,x)$ of~$\eta\lt(\zeta_i(t,x)\rt)$ and~$w_{i}(t,x)$ such that
      \begin{align}
        \label{Lem22_Eq1}
        \lt(\partial_t - \AAA\rt)w_{i}(t,x)=&i\delta \ee^{-\beta t}\lt[\beta\sigma \eta'\lt(\zeta_i(t,x)\rt)-\beta+ F''\lt(\theta^i(t,x)\rt)\rt].
      \end{align}
      
      Recall that~$\Delta_0>\Delta_1$. Then, by~\eqref{Relimite}, there exists~$R_0>0$ such that
      \begin{align*}
	   \lt\{
	   \begin{aligned}
	    &\lt|\eta(y)-\etal\rt|<(\Delta_0-\Delta_1)/2 \si y <-R_0,
	    \\
	    & \lt|\eta(y)-\etar \rt| <(\Delta_0-\Delta_1)/2  \si y >R_0.   
	   \end{aligned}
	    \rt.
      \end{align*}
      Therefore, as~$\sigma \delta<\Delta_1$, we also have
      \begin{align*}
        \lt\{
	  \begin{aligned}
	    &
	    \lt|w_{i}(t,x) -\etal \rt|\leq (\Delta_0+\Delta_1)/2< \Delta_0 &&\si \zeta_i(t,x)<-R_0,\\
	    &
	    \lt|w_{i}(t,x)-\etar\rt|\leq (\Delta_0+\Delta_1)/2< \Delta_0 &&\si \zeta_i(t,x)>R_0.\\
	  \end{aligned}
        \rt.
      \end{align*}
      Let
      \begin{align}
	\label{DefBeta}
        \beta:=\inf\big\{F''(v), |v-\etar| \leq \Delta_0 \quad\text{or}\quad |v-\etal| \leq \Delta_0\big\}
      \end{align}
      Thus, if~$\lt|\zeta_i(t,x)\rt|>R_0$, by definition of~$\beta$ and of~$\theta^i(t,x)$, $F''\lt(\theta^i(t,x)\rt)-\beta \geq 0$. Moreover, $\eta'>0$. Therefore
      \begin{align}
	\label{eq:Estima_A}
        \beta\sigma \eta'\lt(\zeta_i(t,x)\rt)+ F''\lt(\theta^i(t,x)\rt)-\beta \geq 0.
      \end{align}
      Now, we set
      \begin{align}\label{DefSigma}
        \sigma:= \min\lt\{\beta^{-1} \lt( \inf_{|y|<R_0} \eta'(y)\rt)^{-1} \lt(\lt\|F''\rt\|_{\LLinf}+\beta \rt),1 \rt\}.
      \end{align}
      Therefore, if~$\lt|\zeta_i(t,x)\rt| \leq R_0$, we also have~\eqref{eq:Estima_A}.
      As a conclusion, in any case, \eqref{Lem22_Eq1} and~\eqref{eq:Estima_A} yield
      \begin{align*}
        i\lt(\partial_t -\AAA \rt)w_{i}(t,x) \geq 0,
      \end{align*}
      which implies that~$w_{+1}$ and~$w_{-1}$ are respectively a super-solution and a sub-solution to~\eqref{Wd}.
    \end{proof}

    Using Lemma~\ref{Lem22}, it is eventually possible to squeeze a solution~$u(t,x)$ of~\eqref{Wd} between a sub-solution and a super-solution. The following Lemma explains how this squeezing is tightened:

    \begin{lemma}[Lemma 3.3 of~\cite{XinfuChen}] \label{Lem33}
     Under the hypotheses of Theorem~\ref{ThCv}, let~$\Delta_1<\Delta_0$.
      Assume that there exist~$\xi \in \R$, $\delta \in (0,\Delta_1)$ and~$l \in [0,L]$ for fixed~$L$ such that, for all~$x \in \R$,
      \begin{align}
        \label{Lem33_Eq1}
        \eta(x-l) - \delta  \leq u(0,x) \leq \eta(x+l)+\delta.
      \end{align}
      Then, taking~$\beta$ and~$\sigma$ as in Lemma~\ref{Lem22}, there exist a positive constant~$\epsilon_*$ depending only on~$\eta$, $F'$, and~$L$, and parameters~$\tilde\xi$, $\tilde\delta$, $\tilde{l}$ satisfying
      \begin{align*}
	&\lt|\tilde{\xi}\rt| \leq l, \quad \tilde{\delta}= \delta+\epsilon_* \min(1,l) \ee^{\beta},
        \et 0\leq \tilde{l}\leq l+\sigma\delta-\frac{\sigma\epsilon_* \min(1,l)}{2},
      \end{align*}
      such that, for all~$t \geq 1$ and~$x \in \R$,
      \begin{align}
	    \label{Lem33_Res}
	    \eta(x-\tilde\xi-\tilde{l}-ct) -\tilde{\delta} \ee^{-\beta t}\leq u(t,x) \leq \eta(x-\tilde\xi+\tilde{l} -ct)+\tilde{\delta}  \ee^{-\beta t}.
      \end{align}
    \end{lemma}

    \begin{proof}
	Thanks to Lemma~\ref{Lem22}, the functions~$w_{+1}$ and~$w_{-1}$ defined by~\eqref{eq:Def_w}, for~$\zeta_i$  defined by~\eqref{DefZeta},
	are respectively a super-solution and a sub-solution to~\eqref{Wd}. Using~\eqref{Lem33_Eq1}, it follows from Proposition~\ref{PropCompare} that, for all~$t \geq 0$ and~$x \in \R$,
	\begin{align}\label{Compare2}
	    w_{-1}(t,x) \leq u(t,x) \leq w_{+1}(t,x).
	\end{align}

	Let~$\barl:=\min(1,l)$ and~$\epsilon_1:=\inf_{x \in [-1,2]} \eta'(x)$. Since~$\eta$ is increasing, a Taylor expansion yields
	\begin{align*}
	    \int_{0}^1 \lt(\eta(x+l) - \eta(x-l)\rt) \dd x \geq  \int_0^1 \lt(\eta\lt(x+\barl\rt) - \eta\lt(x-\barl\rt)\rt) \dd x \geq  2\epsilon_1 \barl.
	\end{align*}
	Therefore, at least one of the following estimates is true
	\begin{align*}
	    \int_{0}^1 \lt(u(0,x)-\eta(x-l)\rt) \dd x \geq \epsilon_1 \barl 
	    \quad \text{or} \quad 
	    \int_{0}^1 \lt(\eta(x+l)-u(0,x)\rt) \dd x \geq \epsilon_1\barl.
	\end{align*}
	Hereafter, we only consider the first case, as the second one is similar. 
	First, using~\eqref{O21}, there exists~$R_1$ such that
	\begin{align}
	  \label{DefR02}
	    2\sigma \eta'(x) \leq 1 \si |x|>R_1
	\end{align}
	Let~$R_2 :=R_1+L+|c|+1+\sigma\Delta_0$.
	On the one hand, invoking Proposition~\ref{PropCompare},
	we compare~$u$ and~$w_{-1}$ on~$[-R_2 ,R_2]$
	\begin{align}
	    \inf_{x \in [-R_2,R_2]}\lt\{ u(1,x)-  \eta\lt(\zeta_{-1}(1,x)\rt) + \delta \ee^{-\beta}\rt\}
	    &\geq \rho(R_2) \int_{0}^1\lt[u(0,y)-\eta(y-l)+\delta  \rt] \dd y
	    \nonumber
	    \\
	    &\geq \rho(R_2) \epsilon_1 \barl.
	    \label{Lem33_Eq2}
	\end{align}
	We define
	\begin{align}
	    \label{DefEpsilonStar}
	    \epsilon_* := \min\lt(\Delta_1\lt(1-\ee^{-\beta} \rt) , \frac{1}{2\sigma},\frac{\rho(R_2)\epsilon_1}{2\sigma} \lt\|\eta'\rt\|_{\LLinf}^{-1} \rt).
	\end{align}
	 As a consequence, if~$|x|<R_2$, \eqref{Lem33_Eq2} yields
	\begin{align}
	    u(1,x) - \eta\lt(\zeta_{-1}(1,x)+ 2\epsilon_*  \sigma \barl\rt) + \delta \ee^{-\beta}
	    &\geq \rho(R_2)\epsilon_1 \barl - 2\epsilon_* \sigma \barl\lt\|\eta'\rt\|_{\LLinf}
	    \label{Estim1_Boucle}
	    \geq 0.
	\end{align}
	On the other hand, if~$|x|>R_2$, then ~$|\zeta_{-1}(1,x)|\geq R_1+1$ whence, by definition of~$\epsilon^*$, $\lt|\zeta_{-1}(1,x)+2\epsilon_*\sigma \barl\rt| \geq R_1$.
	Inequality~\eqref{Compare2} and Definition~\eqref{DefR02} then imply that
	\begin{align}
	    u(1,x)-\eta\lt(\zeta_{-1}(1,x)+2\epsilon_* \sigma \barl\rt) +\delta \ee^{-\beta} 
	    & \geq u(1,x) -w_{-1}(1,x) -\epsilon_* \barl
	    \label{Estim2_Boucle}
	    &\geq -\epsilon_* \barl.
	\end{align}
	Therefore, from~\eqref{Estim1_Boucle} and from~\eqref{Estim2_Boucle}, it appears that, for all~$x \in \R$, there holds
	\begin{align*}
	    u(1,x) \geq  \eta\lt(\zeta_{-1}(1,x)+2\epsilon_*\sigma \barl \rt) - \lt( \epsilon_* \barl+\delta \ee^{-\beta}\rt).
	\end{align*}
	We set
	\begin{align*}
	  \tilde{\delta}:= \lt(\delta\ee^{-\beta}+\epsilon_* \barl\rt) \ee^{\beta},
	\end{align*}
	which, thanks to~\eqref{DefEpsilonStar}, satisfies~$\tilde{\delta} \ee^{-\beta}\leq \Delta_1$.
	Applying once more Lemma \ref{Lem22} yields, for all~$t \geq 1$,
	\begin{align}
	  \label{Controllu}
	    u(t,x) \geq \eta\lt(\zeta_{-1}(1,x)-c(t-1) 
	    +2\epsilon_* \sigma \barl 
	    - \sigma \tilde{\delta} \ee^{-\beta} \lt[1 -\ee^{-\beta (t-1)} \rt] \rt) - \tilde{\delta} \ee^{-\beta t}.
	\end{align}
	By definition of~$\tilde{\delta}$ and~$\zeta_{-1}$, the argument of~$\eta$ in the above estimate is
	\begin{align}
	  \nonumber
	  &x-ct-l-\sigma\delta\lt[1-\ee^{-\beta}\rt] +2\epsilon_* \sigma \barl 
	  -\sigma \tilde{\delta} \ee^{-\beta} \lt[1 -\ee^{-\beta (t-1)} \rt]
	  \\
	  &\geq x - ct -\lt[l + \sigma \delta-\sigma\epsilon_* \barl \rt].
	  \label{Croissance}
	\end{align}
	Defining now
	\begin{align*}
	    &\quad \tilde{\xi} :=- \frac{\sigma\epsilon_* \barl}{2}, \et
	    \tilde{l} := l +\sigma \delta - \frac{ \sigma \epsilon_*\barl}{2},
	\end{align*}
	and bearing in mind that~$\eta$ is increasing, we deduce from~\eqref{Controllu} and~\eqref{Croissance} that
	\begin{align}
	   \label{Lem33_Res1}
	    u(t,x)
	    \geq&  \eta\lt( x - ct - \tilde{\xi}-\tilde{l}\rt) - \tilde{\delta} \ee^{-\beta t}.
	\end{align}
	Moreover, recalling~\eqref{Compare2}, we have
	\begin{align}
	   \label{Lem33_Res2}
	    u(t,x) \leq & \eta\lt(x-ct+l + \sigma \delta  \rt) + \delta \ee^{-\beta t}
	    \leq \eta\lt(x-ct+\tilde{l} -\tilde{\xi}  \rt) + \tilde{\delta} \ee^{-\beta t}.
	\end{align}
	As a consequence, we obtain the desired result~\eqref{Lem33_Res} from~\eqref{Lem33_Res1} and~\eqref{Lem33_Res2}.
    \end{proof}

    We are now in position to finish the proof of Theorem~\ref{ThCv}. The proof is done while iterating Lemma~\ref{Lem33}, which gradually tightens the squeezing around~$u(t,x)$.

    \begin{proof}[Proof of  (iii) of Theorem~\ref{ThCv} (restriction of the proof of Theorem 3.1 of~\cite{XinfuChen})]
      We proceed in four steps, lowering iteratively in time the values ~$\delta$ and~$l$ such that, for all~$x \in \R$, there holds
      \begin{align}
        \label{Compare_Base}
        \eta\lt(x-ct - \xi- l\rt)-\delta \leq u(t,x) \leq \eta\lt(x-ct-\xi+l\rt)+\delta.
      \end{align}
      
      \paragraph{Step 1}
	By assumption~\eqref{Puits2} and since~$\eta$ is increasing from~$\etal$ to~$\etar$, there exist~$\Delta_1<\Delta_0$, and~$L>1$ sufficiently large such that~\eqref{Compare_Base} holds with
	\begin{align*}
	  t=t_1:=0,  \quad \delta=\delta_1 := \Delta_1, \quad \xi=\xi_1:=0, \et l=l_1:=L-\sigma\Delta_0.
	\end{align*}
      
      \paragraph{Step 2}
	We define
	\begin{align}
	  \label{DefKappaStar}
	  \delta_*:=\min\lt(\Delta_1,\epsilon_*/4\rt) \et \kappa_*:=\sigma \epsilon_*/2 - \sigma \delta_* \geq \sigma \epsilon_*/4 >0.
	\end{align}
	Also, we set~$t_* \geq 2$ such that
	\begin{align}
	  \label{DeftStar}
	  \ee^{-\beta t_*} \lt(1+\epsilon_*/\delta_*\rt) \ee^{\beta} \leq 1-\kappa_*.
	\end{align}
	Using Lemma~\ref{Lem22}, we deduce from the previous step that there exists~$t_2$ sufficiently high such that~\eqref{Compare_Base} holds with~$t=t_2$, $\delta=\delta_2=\delta_*$ and for a certain~$\xi \in \R$ and~$l=l_2 \leq L$ (as~$\epsilon_*$ implicitly depends on~$L$, we further ensures that~$l_j \leq L$, for all~$j$).
	
	If~$l_2\leq 1$, one directly goes to Step 3. Otherwise, as long as~$l_j>1$, one applies Lemma~\ref{Lem33} at time~$t_j=t_2+(j-2)t_*$ (recall that \eqref{Wd} is an autonomous evolution equation), $\delta_j=\delta_*$ and get, by~\eqref{DefKappaStar} and~\eqref{DeftStar}, that~\eqref{Compare_Base} holds for~$t=t_{j+1}$ with
	 ~$l \leq l_{j} - \kappa_*$ and~$\delta \leq (1-\kappa_*)\delta_*$.
	Therefore, one can take~$\delta_{j+1}:=\delta_*$, $l_{j+1}:= l_{j}-\kappa_*$ and iterate until~$l_j<1$.
	
      \paragraph{Step 3}
	By Step 2, we have an index~$j_0$ such that~\eqref{Compare_Base} holds for~$t=t_{j_0}$, $\delta=\delta_{j_0}=\delta_*$, $\xi_{j_0} \in \R$ and~$l=l_{j_0}= 1$. Using Lemma~\ref{Lem33} and Definitions~\eqref{DefKappaStar} and~\eqref{DeftStar}, a straightforward computation inductively shows that,  for all~$j \geq 0$, Inequalities \eqref{Compare_Base} hold  for~$t=t_{j+j_0}$, $\delta=\delta_{j+j_0}$, and~$l=l_{j+j_0}$ being defined by
	\begin{align}
	  \label{HypoRec}
	  t_{j+j_0}:=t_{j_0}+jt_*, \quad \delta_{j_0+j}:=(1-\kappa_*)^j\delta_* \et l_{j+j_0}:=(1-\kappa_*)^j,
	\end{align}
	and for~$\xi = \xi_j$ such that
	 ~$\lt|\xi_{j_0+j+1}-\xi_{j_0+j} \rt| \leq (1-\kappa_*)^{j}$.
	
      \paragraph{Step 4}
	We have shown that~\eqref{Compare_Base} holds for~$\lt(t,\delta,l\rt)=\lt(t_j,\delta_j,l_j\rt)$, for all~$j \geq 0$. 
	For~$t>0$, we associate~$j$ implicitly defined by~$t \in [t_j,t_{j+1})$.
	Thus, we deduce from Lemma~\ref{Lem22} that~\eqref{Compare_Base} also holds for~$t$, $\delta=\delta_j$, $\xi=\xi_j$, and~$l=l_j+\sigma\delta_j$.
	Taking Step 3 into account yields,  for all~$j>j_0$,
	\begin{align*}
	  &\delta  \leq \delta_* (1-\kappa_*)^{j-j_0} \et l  \leq \lt(1+\sigma \delta_*\rt) (1-\kappa_*)^{j-j_0}.
	\end{align*}
	Moreover, $\xi_j$ converges to~$\xi_{\infty}$ and, for all~$j>j_0$,
	\begin{align*}
	  \lt|\xi_j-\xi_{\infty}\rt| \leq \kappa_*^{-1} (1-\kappa_*)^{j-j_0}.
	\end{align*}
	Yet, a simple calculation shows
	\begin{align*}
	  (1-\kappa_*)^{j-j_0}= \lt(1-\kappa_* \rt)^{-j_0-t_{j_0}/t_*}\exp \lt(t\ln(1-\kappa_*)/t_* \rt).
	\end{align*}	
	Setting~$\kappa:=-\ln(1-\kappa_*)/t_*>0$ and
	\begin{align*}
	   &K:=\lt[\delta_* + \lt(1+\sigma \delta_*+ \kappa_*^{-1} \rt)\lt\|\eta'\rt\|_{\LLinf}\rt]  \lt(1-\kappa_* \rt)^{-j_0-t_{j_0}/t_*},
	\end{align*}
	we obtain, for all~$t\geq t_{j_0}$, $t \in [t_j,t_{j+1})$,
	\begin{align*}
	  \sup_{x \in \R} \lt|\eta(x-ct-\xi_{\infty})-u(t,x) \rt| 
	  \leq& \delta_j + \lt(\lt|l_j\rt|+\lt|\xi_j-\xi_\infty\rt|\rt) \lt\|\eta'\rt\|_{\LLinf}
	  \leq K \ee^{-\kappa t}.
	\end{align*}
	This concludes the proof of Theorem~\ref{ThCv}.
    \end{proof}

      \paragraph{Acknowledgements}
	We would like to thank Claude Le Bris for his advice, and Yves-Patrick Pellegrini for his kindness and for providing a physical insight into the Weertman equation.
      
\bibliographystyle{plain}
\def\cprime{$'$}

\end{document}